\documentclass{scrartcl}

\usepackage{natbib}
\usepackage{geometry}
\usepackage{fleqn}
\usepackage{graphicx}
\usepackage[colorlinks,citecolor=blue,urlcolor=blue]{hyperref}
\usepackage{xspace,enumerate} 
\usepackage{amsmath,amssymb}
\usepackage{enumitem}
\usepackage{cite}
\usepackage{mathtools}
\usepackage{dsfont,mathrsfs}
\usepackage{amsthm}
\usepackage{bbm}

\numberwithin{equation}{section}

\theoremstyle{plain}
\newtheorem{theorem}{Theorem}[section]
\numberwithin{theorem}{section}
\newtheorem{corollary}[theorem]{Corollary}

\newtheorem{proposition}[theorem]{Proposition}
\newtheorem{definition}[theorem]{Definition}
\newtheorem{example}[theorem]{Example}
\newtheorem{remark}[theorem]{Remark}

\newcommand{\cD}{\mathcal{D}}
\newcommand{\cF}{\mathcal{F}}
\newcommand{\NN}{\mathbb{N}}
\newcommand{\RR}{\mathbb{R}}
\newcommand{\ind}{\mathbbm{1}}

\DeclareMathOperator{\supp}{supp}
\DeclareMathOperator{\id}{id}

\title{Comparison of path-dependent functionals of semimartingales}
\author{Benedikt K\"opfer\footnote{A LGFG grant of the state Baden-W\"urttemberg is gratefully acknowledged}, Ludger R\"uschendorf}
\date{}

\begin{document}

\maketitle
\thispagestyle{empty}

\begin{abstract}
Based on an extension of the martingale comparison method some comparison results for path-dependent functions of semimartingales are established. The proof makes essential use of the functional It\^o calculus. A main tool is an extension of the Kolmogorov backwards equation to path-dependent functions. The paper also derives criteria for the regularity conditions of the comparison theorems and discusses applications as to the comparison of Asian options for semimartingale models.
\end{abstract}

\renewcommand{\thefootnote}{}
\footnotetext{\hspace*{-.51cm}%
AMS 2010 subject classification:
Primary: 60E15; secondary: 60E44, 60G51\\
Key words and phrases: Path-dependent ordering, ordering of semimartingales, Kolmogorov backwards equation, functional It\^o calculus.}

{\renewcommand{\thefootnote}{\arabic{footnote}}

\section{Introduction}
\label{sec:intro}

The main subject of this paper is to give an extension of ordering results for path-independent functions of semimartingales based on the martingale method to path-dependent functions. The martingale comparison method was introduced for the comparison of path-independent functions of semimartingales in \citet{KJS98} and \citet{BJ00}. It was then systematized and extended in \citet{GM02}, \citet{BR06,BR07a,BR07b,BR08} and in \citet{KR19}. Essentially a comparison of local (differential) characteristics and the 'propagation of order' property yield, under the condition that the propagation operator (the value process) satisfies a Kolmogorov backwards equation, a comparison of terminal values.

In particular in \citet{BR06,BR07a} and \citet{KR19} general versions of the Kolmogorov backwards equation for path-independent functionals have been established and applied to ordering results for semimartingales w.r.t. various kinds of orderings as motivated by the problem to establish price bounds resp. risk bounds in some general class of insurance resp. financial models. Some alternative approaches to related comparison results are given in \citet{GM94}, \citet{KPQ97}, \citet{Ho98}, \citet{Zh04}, \citet{SGL05}, \citet{PZ06}, \citet{Kl06}, \citet{Ar08}, \citet{WX09}, \citet{Ma10} and \citet{Cr19}.

For the extension to the ordering of path-dependent functions we make essential use of the functional It\^o calculus and in particular of the functional It\^o formula, see \citet{BC16}. In Section \ref{sec:funcito} some necessary notions and results of this theory are collected. The functional It\^o formula allows us to extend the basic Kolmogorov backward equation to the path-dependent framework. As a consequence we are able to derive comparison results for path-dependent functions under equivalent martingale measures as well as w.r.t. semimartingale measures. We also discuss the regularity conditions of the comparison theorems and discuss applications as to the the comparison of Asian options for semimartingales. For further details and extensions of the comparison method we refer to the dissertation \citet{Ko19} on which this paper is based.

\section{Functional It\^o calculus}
\label{sec:funcito}

In this section we recall some of the basic notions and results of the functional It\^o calculus. This is the main tool for the extension of the martingale comparison method, to the frame of path-dependent functionals. The functional It\^o calculus was introduced by \citet{Du09} and developed since then; see \citet{CF10a,CF10b}, \citet{Le13}, \citet{BC16} and \citet{AC17}. A comprehensive presentation  on which this section is based is given in \citet{BC16}.

For the functional calculus a set of suitable functions and an appropriate notion of derivative is needed. Let $X$ be the canonical process on the space of c\`adl\`ag functions $\Omega = D([0,T],\RR^d)$ and $(\mathcal{F}_t)_{t \in [0,T]}$ be the filtration generated by it. Then any adapted real-valued process $Y = (Y_t)_{t \in [0,T]}$ may be represented as family of functionals $Y(t,\cdot): \Omega \to \RR$, such that $Y(t,\cdot)$ only depends on the path stopped at $t$, i.e. $Y(t,\omega) = Y(t,\omega_{\cdot \wedge t})$. Therefore we can view an adapted process as functional on the space of ``stopped paths''. In the sequel we use the notation $\omega^t_\cdot := \omega_{\cdot \wedge t}$ for the path stopped at $t$. More formally a stopped path is an equivalence class in $([0,T] \times D([0,T], \RR^d)$ for the following equivalence relation
\begin{align*}
(t,\omega) \sim (s, \tilde{\omega}) \Leftrightarrow t = s~ \text{and}~ \omega^t = \tilde{\omega}^s.
\end{align*}
The space of \emph{stopped paths} is defined as the quotient of $[0,T] \times D([0,T], \RR^d)$ by the above equivalence relation:
\begin{align*}
\Lambda^d_T := \{ (t, \omega^t); (t, \omega) \in [0,T] \times D([0,T], \RR^d) \} = [0,T] \times D([0,T], \RR^d)/\sim.
\end{align*}
The space of stopped paths is a complete metric space for the metric
\begin{align*}
d_\infty((t, \omega^t),(s, \tilde{\omega}^s)) :=&~ \sup_{u \in [0,T]} |\omega_{u \wedge t} - \tilde{\omega}_{u \wedge s}| + |t - s|\\
=&~ \| \omega^t - \tilde{\omega}^{s} \|_\infty + |t - s|.
\end{align*}
In the sequel we write $(t,\omega)$ since it is clear from the first variable at which point in time the path is stopped. If the path is stopped at a certain point prior to $t$ or if we want to emphasize that the path runs until $t$, we use the notation $(t,\omega^t)$.

The class of \emph{non-anticipative functionals} is defined as follows: A non-anticipative functional on $D([0,T], \RR^d)$ is a measurable map $F: (\Lambda^d_T, d_\infty) \to (\RR,\mathfrak{B}(\RR))$. The notion ``non-anticipative'' describes a functional on the path space which only depends of past values. As mentioned in \citet{BC16}, every progressively measurable process can be represented as a non-anticipative functional and conversely.

To define a suitable class of non-anticipative functionals for a path dependent It\^o formula, some regularity properties are needed, in particular the notion of continuity. Continuity of a non-anticipative functional $F: \Lambda^d_T \to \RR$ is defined as continuity as function between the metric spaces $(\Lambda^d_T,d_\infty)$ and $(\RR,|\cdot|)$. Let $C^{0,0} (\Lambda^d_T)$ denote the set of all continuous non-anticipative functionals. A weaker concept is \emph{continuity at fixed times}, i. e. for all $t \in [0,T)$ the map $F(t,.) : (D([0,T], \RR^d), \| \cdot \|_\infty) \to (RR,|\cdot|)$ is continuous. $F$ is called \emph{left-continuous} if $F$ is continuous at fixed times and the following holds
\begin{align*}
&\forall (t,\omega) \in \Lambda^d_T,~ \forall \varepsilon > 0, \exists \delta > 0,~ \forall (s,\tilde{\omega}) \in \Lambda^d_T,\\
&\big[ s < t~ \text{and}~ d_\infty((t, \omega), (s,\tilde{\omega})) < \delta \big] \Rightarrow |F(t,\omega) - F(s,\tilde{\omega})| < \delta.
\end{align*}
The set of all left-continuous non-anticipative functionals is denoted by $C^{0,0}_l(\Lambda^d_T)$.

The property of being \emph{boundedness preserving} is crucial for various results in \citet{BC16} and a precondition for the functional It\^o's formula. A non-anticipative functional $F: \Lambda^d_T \to \RR$ is called boundedness preserving if for any compact $K \subset \RR^d$ and $t_0 < T$ holds
\begin{align*}
\exists C_{K,t_0} > 0, \forall t \le t_0, \forall \omega \in D([0,T],\RR^d), \omega([0,t]) \subset K \Rightarrow |F(t, \omega)| \le C_{K,t_0}.
\end{align*}
Denote by $B(\Lambda^d_T)$ the set of boundedness preserving functionals and by $C^{0,0}_b$ the set of continuous boundedness preserving functionals. 

The derivatives which are used for the functional It\^o calculus are the horizontal and the vertical derivative. For the horizontal derivative, a stopped path $(t,\omega) \in \Lambda^d_T$ is extended to the interval $[0,t+h]$ by its value at time $t$, i.e. to $(t+h,\omega^t)$.

\begin{definition}
\label{def:horizontalderivative}
A non-anticipative functional $F: \Lambda^d_T \to \RR$ is said to be \emph{horizontally differentiable} at $(t,\omega) \in \Lambda^d_T$ if the following limit exists
\begin{align*}
\mathcal{D} F(t,\omega) = \lim_{h \downarrow 0} \frac{F(t+h,\omega^t) - F(t,\omega^t)}{h}.
\end{align*}
If $F$ is horizontally differentiable at all $(t,\omega) \in \Lambda^d_T$, then $\mathcal{D} F$ is a non-anticipative functional, called the \emph{horizontal derivative} of $F$.
\end{definition}

For the vertical derivative, the stopped path at the stopping point is disturbed by a constant $x \in \RR^d$. For a path $\omega \in D([0,T],\RR^d)$ we denote the disturbed path by $\omega^{x,t} := \omega^t + x \ind_{[t,T]}$.

\begin{definition}
\label{def:verticalderivative}
A non-anticipative functional $F: \Lambda^d_T \to \RR$ is said to be \emph{vertically differentiable} at $(t,\omega) \in \Lambda^d_T$ if the map
\begin{align*}
\RR^d &\to \RR\\
x &\mapsto F(t,\omega^{x,t})
\end{align*}
is differentiable in $0$. Its gradient at $0$ is called the \emph{vertical derivative} of $F$ at $(t,\omega)$:
\begin{align*}
\nabla_\omega F(t,\omega) =&~ (\nabla_{\omega i} F(t,\omega), i = 1, \dots, d),
\end{align*}
where for the standard base $(e_i)_{1 \le i \le d}$ of $\RR^d$ the derivatives are defined by
\begin{align*}
\nabla_{\omega i} F(t,\omega) =&~ \lim_{h \to 0} \frac{F(t,\omega^t + h e_i \ind_{[t,T]}) - F(t,\omega^t)}{h}.
\end{align*}
If $F$ is vertically differentiable at all $(t,\omega) \in \Lambda^d_T$, then $\nabla_\omega F$ is a non-anticipative functional called the vertical derivative of $F$. 
\end{definition}

For each $x \in \RR^d$, $\nabla_\omega F(t,\omega).x$ is the directional derivative of $F(t,.)$ in direction $\ind_{[t,T]} x$. As usual one may differentiate multiple times, if possible; we denote this by a superscript, $\nabla_\omega^2, \dots, \nabla_\omega^k$. Note that even if considering only continuous paths, one still has to use $\Lambda_T^d$ for the definition of vertical differentiability to make sense.

For example the non-anticipative functional $F(t,\omega) = f(t,\omega_t)$ with $f \in C^{1,1}([0,T] \times \RR^d)$ has horizontal and vertical derivatives which are simply the partial (right-) derivatives of $f$. Thus, Definitions \ref{def:horizontalderivative} and \ref{def:verticalderivative} are an extension of the notion of partial derivatives.

The next definition introduces a class of regular non-anticipative functionals which is suitable for a path-wise It\^o formula.

\begin{definition}
\label{def:regularitypath}
Define $C^{1,2}_b(\Lambda^d_T)$ as the set of left-continuous non-anticipative functionals $F \in C^{0,0}_l(\Lambda^d_T)$ such that
\begin{itemize}
\item[--] $F$ is horizontally differentiable at all points $(t,\omega) \in \Lambda^d_T$ and $\mathcal{D} F$ is continuous at fixed times;

\item[--] $F$ is twice vertically differentiable and $\nabla_\omega F, \nabla_\omega^2 F \in C^{0,0}_l$;

\item[--] $\mathcal{D} F, \nabla_\omega F, \nabla_\omega^2 F \in B(\Lambda^d_T)$.
\end{itemize}
\end{definition}

In \citet{BC16} it is pointed out that one might use as well right continuity. To apply the pathwise calculus to semimartingales, we use the left-continuity such that the integrands in the pathwise It\^{o} formula are predictable. For the following examples of horizontally and vertically differentaible functionals, see \citet{BC16}.

\begin{example}
\label{ex:exampleregularintroductfuncito}
\begin{enumerate}
\item Let $g \in C^0(\RR^d)$ and $\rho : \RR_+ \to \RR$ be bounded and measurable. Then a non-anticipative functional in $C^{1,\infty}_b(\Lambda^d_T)$ is given by
\begin{align*}
F(t,\omega) := \int_0^t g(\omega_s) \rho(s) ds.
\end{align*}
The horizontal derivative is given by $\cD F(t,\omega) = g(\omega_t) \rho(t)$ and the vertical derivative is $\nabla_{\omega i} F(t,\omega) = 0$. 

\item Let $0 < t_1 < \dots < t_n$ be some points in $[0,T]$, $g \in C^0(\RR^{n \times d})$ and $h \in C^k(\RR^k)$ with $h(0) = 0$. Then 
\begin{align*}
F(t,\omega) = h(\omega_t - \omega_{t_n^-}) \ind_{t \ge t_n} g(\omega_{t_1^-},\omega_{t_2^-}, \dots , \omega_{t_n^-})
\end{align*}
is of class $C_b^{1,k}(\Lambda_T^d)$. The horizontal derivative is $\cD F(t,\omega) = 0$ and the vertical derivative is $\nabla_{\omega i} F(t,\omega) = \partial_i h(\omega_t - \omega_{t_n^-}) \ind_{t \ge t_n} g(\omega_{t_1^-},\omega_{t_2^-}, \dots , \omega_{t_n^-})$.
\end{enumerate}
\end{example}

Definition \ref{def:regularitypath} can be extended by localization.

\begin{definition}
A non-anticipative functional $F \in C^{0,0}_b(\Lambda^d_T)$ is called locally regular if there exists an increasing sequence $(\tau_k)_{k \in \NN}$ of stopping times with $\tau_0 = 0$, $\tau_k \uparrow \infty$ and $F^k \in C^{1,2}_b(\Lambda^d_T)$ such that 
\begin{align*}
F(t,\omega) = \sum_{k \in \NN} F^k(t,\omega) \ind_{[\tau_k,\tau_{k+1})} (t).
\end{align*}
The set of all locally regular functionals is denoted by $C^{1,2}_{loc}(\Lambda^d_T)$.
\end{definition} 

By definition ~$C^{1,2}_b (\Lambda^d_T) \subset C^{1,2}_{loc} (\Lambda^d_T)$; a difference is that there may be discontinuities or explosions at the stopping times of the locally regular non-anticipative functionals.

A main result in \citet{BC16} is a path-dependent It\^o formula for paths of semimartingales.
 
\begin{theorem}[Functional It\^o formula]
Let $X$ be an $\RR^d$-valued semimartingale. Then for all $F \in C^{1,2}_{loc}(\Lambda^d_T)$ and all $t \in [0,T]$ we have almost surely
\begin{align}
\label{eq:pathdepito}
F(t,X^t) - F(0,X^0) =&~ \int_0^t \cD F(s,X^{s^-}) ds + \frac{1}{2} \sum_{1\le i,j \le d} \int_0^t \nabla^2_{\omega i,j} F(s,X^{s^-}) d[X]^{c ij}_s \nonumber\\ 
&~ + \sum_{1 \le i \le d} \int_0^t \nabla_{\omega i} F(s,X^{s^-}) d X^i_s \\
&~+ \sum_{s \in (0,t]} \left[F(s,X^s) - F(s,X^{s^-}) - \sum_{1 \le i \le d} \nabla_{\omega i} F(s,X^{s^-}) \Delta X^i_s \right].\nonumber
\end{align}
\end{theorem}

\begin{remark}
In \citet{BC16} a more general version of the functional It\^o formula is derived. Therefore the quadratic variation along a sequence of partitions and the F\"ollmer integral is used. This is established by a non probabilistic pathwise approach, based on ideas from \citet{Fo81}. In the case of semimartingales this reduces to the quadratic variation and the F\"ollmer integral coincides with the stochastic integral. This implies that the comparison results in our paper can be stated for more general processes, e.g. for fractional processes. However our approach relies on (local) martingale properties and can hence not be transferred directly. 
\end{remark}

\section{Path-dependent comparison of semimartingales}
\label{sec:pathdepcomp}

Based on the functional It\^o formula in this section ordering results are derived for path-dependent functions of semimartingales by an extension of the martingale comparison method for the path-independent case. The first main step is to develop a version of the Kolmogorov backwards equation for path-dependent functions. This equation then allows to derive comparison results under equivalent martingale measures and w.r.t. semimartingale measures using the path-dependent It\^o formula in an essential way.

\subsection{Kolmogorov backwards equation}
\label{subsec:kolmogorov}

In this subsection we establish a path-dependent version of the Kolmogorov backwards equation. Let $X$ be a (special) semimartingale on a filtered probability space $(\Omega, \cF, (\cF_t)_{t \in [0,T]}, P)$. We denote by $\hat{X} = (\id,X)$ the corresponding space-time process. Let $(B,C,\nu)$ be the semimartingale characteristics of $\hat{X}$ under $P$ and denote by $(b,c,K)$ the differential characteristics under $P$ with respect to an increasing process $A$, see \citet{JS03}. We denote by $dA$ the measure associated to $A$ and by a superscript the dimension of the semimartingale. In the sequel we write $X^T$ for the whole path of $X$. For a non-anticipative functional $F \in C^{0,1}_{loc}(\Lambda^d_T)$, we define the increment functional
\begin{align*}
H_F : &\Lambda^d_T \times \RR^d \to \RR,\\
& (t,\omega,x) \mapsto F(t,\omega^{t^-} + x \ind_{[t,T]}) - F(t,\omega^{t^-} ) - \sum_{1 \le i \le d} \nabla_{\omega i} F(t,\omega^{t^-}) x^i.
\end{align*}
The following is a path-dependent version of the \emph{Kolmogorov backwards equation} for the case that the underlying semimartingale is a local martingale.

\begin{proposition}
\label{prop:kolmogorovbackwardpathemm}
Let $F \in C^{1,2}_{loc}(\Lambda^d_T)$ and let $X$ be a local martingale. Assume that:
\begin{itemize}
\item[(i)] $(F(t,X^t))_{t \ge 0}$ is a local martingale,

\item[(ii)] $|H_F| \ast \mu^X \in \mathscr{A}_{loc}^+$;
\end{itemize}
Then the following holds $dA \times P$ almost surely
\begin{align}
\begin{split}
\label{eq:kolmogorovbackwardpathemm}
U_tF(t,X^{t^-}) := &~\cD F(t,X^{t^-}) b_t + \frac{1}{2} \sum_{i,j \le d} \nabla^2_{\omega ij} F(t,X^{t^-}) c^{ij}_t\\
& + \int_{\RR^d} H_F(t, X^{t^-},x) K_t(dx) = 0.
\end{split}
\end{align}
\end{proposition}

\begin{proof}
By It\^{o}'s formula for non-anticipative functionals, $F$ has the following representation
{\small
\begin{align*}
F(t, X^t) = &~ F(0,X^0) + \int_0^t \cD F(s,X^{s^-}) b_s dA_s + \sum_{i \le d} \int_0^t \nabla_{\omega i} F(s, X^{s^-}) dX^i_s\\
&~ + \frac{1}{2} \sum_{i,j \le d} \int_0^t \nabla^2_{\omega ij} F(s,X^{s^-}) c^{ij}_s dA_s\\ 
&~ + \int_{[0,t] \times \RR^d} \left[ F(s, X^{s^-} + x\ind_{[s,T]}) - F(s, X^{s^-}) - \sum_{i \le d} \nabla_{\omega i} F(s, X^{s^-}) x^i \right] \mu^X (ds,dx).
\end{align*}}
We compensate the jump integral and combine the local martingales from the $dX$ integrals and the compensated jump integral to a local martingale $(M_t)_{t \in [0,T]}$. Then we have
\begin{align*}
F(t, X^t) = &~ F(0,X^0) + M_t + \int_0^t \cD F(s,X^{s^-}) b_S dA_s\\ 
&~ + \frac{1}{2} \sum_{i,j \le d} \int_0^t \nabla^2_{\omega ij} F(s,X^{s^-}) c^{ij}_s dA_s + \int_{[0,t] \times \RR^d} H_F(s,X^{s^-},x) K_s(dx)dA_s.
\end{align*}
It follows that the process
\begin{align*}
\int_0^t & \left[ \cD F(s,X^{s^-}) b_s + \frac{1}{2} \sum_{i,j \le d} \nabla^2_{\omega ij} F(s,X^{s^-}) c^{ij}_s + \int_{\RR^d} H_F(s,X^{s^-},x) K_s(dx) \right] dA_s\\
&= \int_0^t U_s F(s,X^{s-}) dA_s.
\end{align*}
is a predictable local martingale of finite variation starting in zero. As consequence this process is almost surely zero. Thus, the integrand has to be $dA \times P$ almost surely zero as well.
\end{proof}

We proceed with the case when $X$ is a special semimartingale, which implies that the process $\hat{X}$ is a special semimartingale as well. Recall that we can use the identity as truncation function and hence the canonical decomposition of $\hat{X}$ has the form:
\begin{align*}
\hat{X}_t = \hat{X}_0 + \left(0,X_t^c + x \ast (\mu^X - \nu)_t \right) + (\hat{b}^X \cdot \hat{A} )_t.
\end{align*}
The following result then states a path-dependent version of the Kolmogorov backwards equation for special semimartingales.

\begin{proposition}
\label{prop:kolmogorovbackwardpathp}
Let $F \in C^{1,2}_{loc}(\Lambda^d_T)$ and let $X$ be a special semimartingale. Assume that:
\begin{itemize}
\item[(i)] $(F(t,X^t))_{t \ge 0}$ is a local martingale;

\item[(ii)] $|H_F| \ast \mu^X \in \mathscr{A}_{loc}^+$;
\end{itemize}

Then the following holds $dA \times P$ almost surely
\begin{align}
\begin{split}
\label{eq:kolmogorovbackwardpathp}
\bar{U}_tF(t,X^{t^-}) :=&~ \cD F(t,X^{t^-}) b_t + \sum_{i \le d} \nabla_{\omega i} F(t, X^{t^-}) b^i_t\\
& + \frac{1}{2} \sum_{i,j \le d} \nabla^2_{\omega ij} F(t,X^{t^-}) c^{ij}_t + \int_{\RR^d} H_F(t,X^{t^-},x) K_t(dx) = 0.
\end{split}
\end{align}
\end{proposition}

\begin{proof}
It\^{o}'s formula for non-anticipative functionals yields
{\small
\begin{align*}
F(t, X^t) = &~ F(0,X^0) + \int_0^t \cD F(s,X^{s^-}) b_s dA_s + \sum_{i \le d} \int_0^t \nabla_{\omega i} F(s, X^{s^-}) dX_s^i\\
&~ + \frac{1}{2} \sum_{i,j \le d} \int_0^t \nabla^2_{\omega ij} F(s,X^{s^-}) c^{ij}_s dA_s\\ 
&~ + \int_{[0,t] \times \RR^d} \left[ F(s, X^{s^-} + x \ind_{[s,T]}) - F(s, X^{s^-}) - \sum_{i \le d} \nabla_{\omega i} F(s, X^{s^-}) x^i \right] \mu^X (ds,dx).
\end{align*}}
We unite the local martingales into one local martingale $M$ as in the proof of Proposition \ref{prop:kolmogorovbackwardpathemm}. Here these are, by the canonical decomposition, the integrals with respect to $X^c$ and with respect to the compensated jump integrals. As a result we obtain
\begin{align*}
F(t, X^t) = &~ F(0,X^0) + M_t + \int_0^t \cD F(s,X^{s^-}) b_s dA_s + \sum_{i \le d} \int_0^t \nabla_{\omega i} F(s, X^{s^-}) b^i_s dA_s\\
&~ + \frac{1}{2} \sum_{i,j \le d} \int_0^t \nabla^2_{\omega ij} F(s,X^{s^-}) c^{ij}_s dA_s + \int_{[0,T] \times \RR^d} H_F(s,X^{s^-},x) K_s(dx) dA_s.
\end{align*}
So the process
\begin{align*}
&~\int_0^t \left[ \cD F(s,X^{s^-}) b_s + \sum_{i \le d} \nabla_{\omega i} F(s, X^{s^-}) b^i_s + \frac{1}{2} \sum_{i,j \le d} \nabla^2_{\omega ij} F(s,X^{s^-}) c^{ij}_s \right.\\
&~ \left. + \int_{\RR^d} H_F(s,X^{s^-},x) K_s(dx) \vphantom{\sum_{i,j \le d} \nabla^2_{\omega ij}} \right] dA_s = \int_0^t \bar{U}_s F(s,X^{s-}) dA_s.
\end{align*}
is a predictable local martingale of finite variation starting in zero implying that it is almost surely zero. Thus, the integrand has to be $dA \times P$ almost surely zero as well.
\end{proof}

\subsection{Comparison results under equivalent martingale measures}

Based on the Kolmogorov backwards equations in Section \ref{subsec:kolmogorov} we derive path-dependent comparison results under e.m.m.. Therefore, let $X$ and $Y$ be semimartingales which possess an e.m.m. each. We denote the e.m.m. and semimartingale characteristics which occur by superscript to make clear to which semimartingale they correspond.

We introduce the path-dependent \emph{propagation operator (valuation functional)}. Therefore, let $f: D([0,T],\RR^d) \to \RR$ be a measurable function then we define the valuation functional $G_f$ by
\begin{align}
\label{eq:backwardfunctionalpath}
G_f(t, \omega) := E_{Q^X}\left [\left .f(X^T)\right |X^t = \omega^t\right ].
\end{align}
This is a non-anticipative functional. Considering
\begin{align*}
G_f(t,X^t) = E_{Q^X}[f(X^T)|\sigma(X_s;s \le t)],
\end{align*}
we see that this functional takes into account the complete past of the semimartingale $X$ and that it is by construction a martingale with respect to the natural filtration generated by $X$. In that case $G_f(t,X^t)$ is a martingale and fulfills equation \eqref{eq:kolmogorovbackwardpathemm}.

Since we need to control the second vertical derivatives, we need the following path-dependent notion of convexity from \citet{Ri15}. 

\begin{definition}
\label{def:definitionverticallyconvex}
A non-anticipative functional $F: \Lambda_T^d \to \RR$ is called \emph{vertically convex} on $U \subset \Lambda_T^d$ if for all $(t,\omega) \in U$ there exists a neighbourhood $V \subset \RR^d$ of $0$ such that the map 
\begin{align}
\begin{split}
\label{eq:definitionverticallyconvex}
V &\to \RR\\
e &\to F\left (t,\omega^t + e \ind_{[t,T]}\right )
\end{split}
\end{align}
is convex.
\end{definition}

For a non-anticipative functional $F \in C^{0,2}(\Lambda_T)$ which is vertically convex it holds that the matrix of the second vertical derivative is positive semidefinite. This follows directly from the definition of the vertical directional derivative in Definition \ref{def:verticalderivative}, and the convexity of the function in \eqref{eq:definitionverticallyconvex}.

In the sequel also vertical directional convexity is a relevant property for the comparison results. We define it analogously to vertical convexity. 

\begin{definition}
A non-anticipative functional $F: \Lambda_T^d \to \RR$ is called \emph{vertically directional convex} on $U \subset \Lambda_T^d$ if for all $(t,\omega) \in U$ there exists a neighbourhood $V \subset \RR^d$ of $0$ such that the map 
\begin{align*}
V &\to \RR\\
e &\to F\left (t,\omega^t + e \ind_{[t,T]}\right )
\end{align*}
is directionally convex.
\end{definition}

For the notion of vertical directional convexity it holds that:\\
$F \in C^{0,2}(\Lambda_T)$ is vertically directional convex on $U$ if and only if $\nabla^2_{\omega ij} F(t,\omega) \ge 0$ for all $i,j \le d$ and all $(t,\omega) \in U$.

\begin{theorem}[Vertical directional convex comparison under e.m.m.] 
\label{thm:comparisonpathemmdcx}
Let $X, Y$ be semimartingales such that $X_0 = Y_0 = x_0$ almost surely and let $f(X^T) \in L^1(Q^X)$, $f(Y^T) \in L^1(Q^Y)$. Assume that
\begin{itemize}
\item[(i)] $G_f \in C_{loc}^{1,2}(\Lambda_T^d)$ and $G_f$ is vertically directional convex on $\Lambda_T^d$;

\item[(ii)] $U_t G_f(t,Y^{t^-}) = 0$ holds $dA \times Q^Y$ almost surely for all $t \in [0,T]$ where the operator $U$ is defined in \eqref{eq:kolmogorovbackwardpathemm} with the differential semimartingale characteristics of $\hat{X}$ under $Q^X$;

\item[(iii)] $|H_{G_f}| \ast \mu^Y \in \mathscr{A}_{loc}^+$;

\item[(iv)] $(G_f(t, Y^t)^-)_{t \in [0,T]}$ is of class (DL);

\item[(v)] $A^{\hat{Y}} = A^{\hat{X}}$;

\item[(vi)] The differential characteristics are $dA^{\hat{Y}} \times Q^Y$ almost surely ordered for all $i,j \le d$; i.e.
\begin{align*}
c_t^{\hat{Y} ij} \le&~ c_t^{\hat{X} ij},\\
\int_{\RR^d} H_{G_f}(t,Y^{t^-},x) K^{\hat{Y}}_t(dx) \le&~ \int_{\RR^d} H_{G_f}(t,Y^{t^-},x) K^{\hat{X}}_t(dx).
\end{align*}
\end{itemize} 

Then it holds that
\begin{align*}
E_{Q^Y}\left [f(Y^T)\right ] \le E_{Q^X}\left [f(X^T)\right ].
\end{align*}
If the inequalities in $(vi)$ are reversed and $(G_f(t, Y^t)^+)_{t \in [0,T]}$ is of class (DL), we have
\begin{align*}
E_{Q^Y}\left [f(Y^T)\right ] \ge E_{Q^X}\left [f(X^T)\right ].
\end{align*} 
\end{theorem}

\begin{proof}
For the proof we establish that the process $(G_f(t, Y^t))_{t \in [0,T]}$ is a $Q^Y$-supermartingale. Then it follows that
\begin{align*}
E_{Q^Y}\left [f(Y^T)\right ] = E_{Q^Y}\left [G_f(T,Y^T)\right ] \le G_f(0,x_0) = E_{Q^X}\left [f(X^T)\right ].
\end{align*}
Since $G_f \in C_{loc}^{1,2}(\Lambda_T^d)$,  we can apply It\^{o}'s formula for non-anticipative functionals and obtain that $(G_f(t, Y^t))_{t \in [0,T]}$ is a semimartingale with decomposition
{\small
\begin{align*}
\MoveEqLeft G_f(t, Y^t)\\
=&~ G_f(0,x_0) + \int_0^t \cD G_f(s,Y^{s^-}) b^{\hat{Y}}_s dA^{\hat{Y}}_s + \sum_{i \le d} \int_0^t \nabla_{\omega i} G_f(s, Y^{s^-}) dY_s^i\\
&~ + \frac{1}{2} \sum_{i,j \le d} \int_0^t \nabla^2_{\omega ij} G_f(s,Y^{s^-}) c^{\hat{Y}ij}_s dA^{\hat{Y}}_s\\ 
&~ + \int_{[0,t] \times \RR^d} \left[ G_f(s, Y^{s^-} + x\ind_{[t,T]}) - G_f(s, Y^{s^-}) - \sum_{i \le d} \nabla_{\omega i} G_f(s, Y^{s^-}) x^i \right] \mu^Y (ds,dx).
\end{align*}}
We compensate the jump integral and combine the local martingales into $M$. Keeping in mind that $Y$ is a $Q^Y$ local martingele, this leads to
{ \small
\begin{align*}
G_f(t, Y^t) = &~ G_f(0,x_0) + M_t + \int_0^t \cD G_f(s,Y^{s^-}) b^{\hat{Y}}_s dA^{\hat{Y}}_s + \frac{1}{2} \sum_{i,j \le d} \int_0^t \nabla^2_{\omega ij} G_f(s,Y^{s^-}) c^{\hat{Y}ij}_s dA^{\hat{Y}}_s\\ 
&~ + \int_{[0,t] \times \RR^d} H_{G_f}(s,Y^{u^-},x) K^{\hat{Y}}_s(dx)dA^{\hat{Y}}_u.
\end{align*}}
To gain the local supermartingale property we show that the following process $(Z_t)$ is decreasing:
{\small
\begin{align*}
Z_t := \int_0^t \left[ \cD G_f(s,Y^{s^-}) b^{\hat{Y}}_s + \frac{1}{2} \sum_{i,j \le d} \nabla^2_{\omega ij} G_f(s,Y^{s^-}) c^{\hat{Y}ij}_s + \int_{\RR^d} H_{G_f}(s,Y^{s^-},x) K^{\hat{Y}}_s(dx) \right]dA^{\hat{Y}}_s.
\end{align*}}
By Assumption~$(v)$ we have that $b^{\hat{Y}}_t dA^{\hat{Y}}_t = b^{\hat{X}}_t dA^{\hat{Y}}_t = dt$. With Assumption~$(ii)$ we obtain
{\small
\begin{align*}
Z_t = \int_0^t \left[\frac{1}{2} \sum_{i,j \le d} \nabla^2_{\omega ij} G_f(s,Y^{s^-}) \left (c^{\hat{Y}ij}_s - c^{\hat{X}ij}_s\right ) + \int_{\RR^d} H_{G_f}(s,Y^{s^-},x) \left (K^{\hat{Y}}_s(dx) - K^{\hat{X}}_s(dx)\right ) \right]dA^{\hat{Y}}_s.
\end{align*}}
Due to the vertical directional convexity and $(vi)$ the first integrand is non-positive. That the second integrand is non-positive follows by Assumption~$(vi)$.\\
Therefore, $-Z \in \mathscr{A}_{loc} ^+$ and $(G_f(t, Y^t))_{t \in [0,T]}$ is a local $Q^Y$-supermartingale.\\
Finally, by Assumption~$(iv)$ follows that $(G_f(t, Y^t))_{t \in [0,T]}$ is a proper $Q^Y$ supermartingale.\\
With reversed inequalities and assuming that $(G_f(t, Y^t)^+)_{t \in [0,T]}$ is of class (DL), we get the submartingale property for $(G_f(t, Y^t))_{t \in [0,T]}$.
\end{proof}

\begin{remark}
\begin{enumerate}
\item Instead of demanding that the kernels are ordered for $H_{G_f}$, we could also have demanded that they are ordered for a bigger function class, for example for all functions which are directionally convex. Note that by vertical directional convexity of $G_f$, $H_{G_f}$ is directionally convex in $x$.

\item Compared to previous papers on this topic we do not need the propagation of order property. The propagation of order means that the propagation operator maps particular function classes, like (directional) convex functions or increasing functions, into themselves. Since we consider a single function we only assume that the propagtion operator maps this function into the class of vertically directional convex functions.
\end{enumerate}
\end{remark}

Next we consider the case that $G_f$ is a vertically convex function.

\begin{theorem}[Vertical convex comparison under e.m.m.]
\label{thm:comparisonpathemmcx}
Let $X,Y$ be semimartingales with $X_0 = Y_0 = x_0$ almost surely and let $f(X^T) \in L^1 (Q^X)$, $ f(Y^T) \in L^1(Q^Y)$. Assume that
\begin{itemize}
\item[(i)] $G_f \in C_{loc}^{1,2}(\Lambda_T^d)$ and $G_f$ is vertically convex;

\item[(ii)]-- $(v)$ of Theorem \ref{thm:comparisonpathemmdcx} hold;

\item[(vi)] The differential characteristics are $dA^{\hat{Y}} \times Q^Y$ almost surely ordered:
\begin{align*}
c_t^{\hat{Y}} \le_{psd}&~ c_t^{\hat{X}},\\
\int_{\RR^d} H_{G_f}(t,Y^{t^-},x) K^{\hat{Y}}_t(dx) \le&~ \int_{\RR^d} H_{G_f}(t,Y^{t^-},x) K^{\hat{X}}_t(dx).
\end{align*}
\end{itemize} 

Then it holds that
\begin{align*}
E_{Q^Y}\left [f(Y^T)\right ] \le E_{Q^X}\left [f(X^T)\right ].
\end{align*}
If the inequalities in $(vi)$ are reversed and $(G_f(t, Y^t)^+)_{t \in [0,T]}$ is of class (DL), we have
\begin{align*}
E_{Q^Y}\left [f(Y^T)\right ] \ge E_{Q^X}\left [f(X^T)\right ].
\end{align*}
\end{theorem}

\begin{proof}
We show that $(G_f(t,Y^t))_{t \in[0,T]}$ is a $Q^Y$-supermartingale. Analogously to the proof of Theorem \ref{thm:comparisonpathemmdcx} we need to show, that $dA^{\hat{Y}} \times Q^Y$ a.s.
\begin{align}
\label{eq:nonpos}
\begin{split}
\frac{1}{2} \sum_{i,j \le d} \nabla^2_{\omega ij} G_f(s,Y^{s^-})& \left (c^{\hat{Y}ij}_s - c^{\hat{X}ij}_s\right )\\
&+ \int_{\RR^d} H_{G_f}(s,Y^{s^-},x) \left (K^{\hat{Y}}_s(dx) - K^{\hat{X}}_s(dx)\right ) \le 0.
\end{split}
\end{align}
Then the assertion follows since the other terms in the functional It\^o formula are local martingales. As in the proof of Theorem 3.4 in \citet{KR19} we get by positive definiteness, that the eigendecomposition of the matrix $-(c^{\hat{Y}}_s - c^{\hat{X}}_s) = c_s^{\hat{X}} - c_s^{\hat{Y}}$ has the form $(\sum_{k \le d} \lambda_k e_k^i e_k^j)_{i,j \le d}$ with eigenvalues $\lambda_k \ge 0$ and eigenvectors $e_k$. We obtain equality of the first process above with
\begin{align*}
- \frac{1}{2} \sum_{k \le d} \lambda_k \sum_{i,j \le d} \nabla^2_{\omega ij} G_f(s,Y^{s^-}) e^i_k e^j_k = - \frac{1}{2} \sum_{k \le d} \lambda_k e'_k \nabla_\omega^2 G_f(s,Y^{s^-}) e_k,
\end{align*}
which is non-positive $dA^{\hat{Y}} \times Q^Y$ almost surely due to the positive semidefiniteness of the matrix $\nabla_\omega^2 G_f$. \\
The second integrand is non-positive $dA^{\hat{Y}} \times Q^Y$ almost surely by Assumption~$(vi)$. With Assumption~$(iv)$ it follows that $(G_f(t,Y^t))_{t \in[0,T]}$ is a proper supermartingale.\\
If the inequalities in $(vi)$ are reversed and $(G_f(t, Y^t)^+)_{t \in [0,T]}$ is of class (DL), we have that $(G_f(t,Y^t))_{t \in[0,T]}$ is a submartingale.
\end{proof}

With the help of the key inequality of the proofs above, we can state a corollary which does not need the assumption of vertical convexity or vertical directional convexity but only uses the inequality in \eqref{eq:nonpos} for a comparison result.

\begin{corollary}[General comparison condition under e.m.m.]
Let $X,Y$ be semimartingales and let $X_0 = Y_0 = x_0$ almost surely. Further let $f$ be such that $f(X^T) \in L^1(Q^X)$ and $f(Y^T) \in L^1(Q^Y)$. Assume that $G_f \in C_{loc}^{1,2}(\Lambda_T^d)$ and that Assumptions~$(ii)$--$(v)$ of Theorem \ref{thm:comparisonpathemmdcx} hold. Further, let $dA^{\hat{Y}} \times Q^Y$ almost surely inequality \eqref{eq:nonpos} hold. Then we obtain
\begin{align*}
E_{Q^Y}\left [f(Y^T)\right ] \le E_{Q^X}\left [f(X^T)\right ]
\end{align*}
If the inequality is reversed and $(G_f(t, Y^t)^+)_{t \in [0,T]}$ is of class (DL), then we obtain
\begin{align*}
E_{Q^Y}\left [f(Y^T)\right ] \ge E_{Q^X}\left [f(X^T)\right ].
\end{align*}
\end{corollary}

\begin{proof}
The process $Z$ from the proof of Theorem \ref{thm:comparisonpathemmdcx} is by inequality \eqref{eq:nonpos} decreasing and hence $G_f$ is a supermartingale. The inverse inequality follows since $Z$ then is increasing and hence $G_f$ is a submartingale.
\end{proof}

The Girsanov transform can be used to compare the expectation under different e.m.m. This leads to the path-dependent version of Corollary 3.8 in \citet{KR19}. By Girsanov's theorem only the compensator of the jump measure changes, the predictable quadratic variation of the continuous martingale part and the increasing process of a good version of the semimartingale characteristics remain the same, cf. \citet[Theorem III.3.24]{JS03}. 

\begin{corollary}[Comparison of e.m.m.]
\label{cor:girsanovpathemm}
Let $X$ be a semimartingale. Let $Q^1$ and $Q^2$ be equivalent local martingale measures for $X$. We denote the particular semimartingale characteristics of $X$ by superscript. Assume that $f(X^T) \in L^1(Q^1) \cap L^1(Q^2)$ and that
\begin{enumerate}
\item[(i)] $G_f \in C_{loc}^{1,2}(\Lambda_T^d)$,

\item[(ii)] $U_t^X G_f(t,X^{t^-}) = 0$, $dA^{\hat{X}} \times Q^1$ almost surely where $U_t^X$ is defined as in \eqref{eq:kolmogorovbackwardpathemm} with semimartingale characteristics of $X$ under $Q^2$;

\item[(iii)] $\big| H_{G_f} \big| \ast \mu^X \in \mathscr{A}_{loc}^+$;

\item[(iv)] $(G_f(t, X^t)^-)_{t \in [0,T]}$ is of class (DL);

\item[(v)] The kernels $K^1$ and $K^2$ are $dA^{\hat{X}} \times Q^1$ almost surely ordered for all $t \in [0,T]$:
\begin{align*}
\int_{\RR^d} H_{G_f}(t,X_{t^-},x) K^1_t(dx) \le&~ \int_{\RR^d} H_{G_f}(t,X_{t^-},x) K^2_t(dx).
\end{align*}
\end{enumerate} 

Then it holds
\begin{align*}
E_{Q^1}\left [f(X^T)\right ] \le E_{Q^2}\left [f(X^T)\right ].
\end{align*}
If the inequality in $(v)$ is reversed and $(G_f(t, X^t)^+)_{t \in [0,T]}$ is of class (DL), then:
\begin{align*}
E_{Q^1}\left [f(X^T)\right ] \ge E_{Q^2}\left [f(X^T)\right ].
\end{align*} 
\end{corollary}

\begin{proof}
This follows with help of the functional It\^{o} formula in a similar way as in the path independent case replacing the horizontal derivative of $G_f$ by the vertical derivatives. This replacement is possible by Assumption~$(ii)$.
\end{proof}

\subsection{Comparison results under the semimartingale measure P}

The following results are versions of Theorems \ref{thm:comparisonpathemmdcx} and \ref{thm:comparisonpathemmcx} under the semimartingale measure $P$. Let $X$ and $Y$ be special semimartingales. Then the space-time processes $\hat{X}$ and $\hat{Y}$ are special semimartingales and we can choose for both semimartingales the same integrator process $A$ for a good version of the semimartingale characteristics, for details see \citet[Section 4.4.]{Ko19}.

We adapt the non-anticipative value functional $G_f$ from equation \eqref{eq:backwardfunctionalpath} to $P$: 
\begin{align*}
G_f (t,\omega) := E[f(X^T)|X^t=\omega^t].
\end{align*}
In the path-independent comparison under $P$ in \citet{KR19} it is assumed that $G_f(t,\cdot)$ is an increasing function for all $t \in [0,T]$ in order to control the first partial derivative. To control the first vertical derivative of non-anticipative functionals we introduce vertical monotonicity.

\begin{definition}
A non-anticipative functional $F: \Lambda_T^d \to \RR$ is called \emph{vertically monotone} on $U \subset \Lambda_T^d$ if for all $(t,\omega) \in U$ there exists a neighbourhood $V \subset \RR^d$ of \/$0$ such that the map 
\begin{align*}
V &\to \RR\\
e &\to F(t,\omega^t + e \ind_{[t,T]})
\end{align*}
is monotone in $e$. 
\end{definition}

This definition guarantees that the first vertical derivative is non-negative or non-positive if it exists. 

\begin{theorem}[Vertically increasing and vertically directional convex comparison under P]
\label{thm:orderingpathidcxp}
Let $X,Y$ be special semimartingales and let $X_0 = Y_0 = x_0$ almost surely. Consider a function $f \in L^1(P^{X^T}) \cap L^1(P^{Y^T})$ and assume that
\begin{itemize}
\item[(i)] $G_f \in C_{loc}^{1,2}(\Lambda_T^d)$ and $G_f$ is vertically directionally convex and vertically increasing on $\Lambda_T^d$;

\item[(ii)] $\bar{U}_t G_f(t,Y^{t^-}) = 0$ holds $dA \times P$ almost surely for all $t \in [0,T]$, where $\bar{U}$ is defined as in \eqref{eq:kolmogorovbackwardpathp} with the characteristics of $\hat{X}$;

\item[(iii)] $|H_{G_f} | \ast \mu^Y \in \mathscr{A}_{loc}^+$;

\item[(iv)] $(G_f(t, Y^t)^-)_{t \in [0,T]}$ is of class (DL);

\item[(v)] The differential characteristics are $dA \times P$ almost surely ordered:
\begin{align*}
b_t^{\hat{Y} i} \le&~ b_t^{\hat{X} i},\\
c_t^{\hat{Y} ij} \le&~ c_t^{\hat{X} ij},\\
\int_{\RR^d} H_{G_f}(t,Y^{t^-},x) K^{\hat{Y}}_t(dx) \le&~ \int_{\RR^d} H_{G_f}(t,Y^{t^-},x) K^{\hat{X}}_t(dx).
\end{align*}
\end{itemize} 

Then it holds:
\begin{align*}
E[f(Y^T)] \le E[f(X^T)].
\end{align*}
If the inequalities in $(v)$ are reversed and $(G_f(t, Y^t)^+)_{t \in [0,T]}$ is of class (DL), we get 
\begin{align*}
E[f(Y^T)] \ge E[f(X^T)].
\end{align*}
\end{theorem}

\begin{proof}
Analogously to the comparison under equivalent martingale measures we establish that $(G_f(t,Y^t))_{t \in [0,T]}$ is a supermartingale. Therefore, using the functional It\^o formula we have to verify that $dA \times P$ almost surely it holds
\begin{align*}
\sum_{i \le d} \nabla_{\omega i} G_f(s, Y^{s^-}) \left(b^{\hat{Y} i}_s - b^{\hat{X}i}_s \right) + \frac{1}{2} \sum_{i,j \le d} \nabla^2_{\omega ij} G_f(s,Y^{s^-}) \left (c^{\hat{Y}ij}_s - c^{\hat{X}ij}_s\right )\\
+ \int_{\RR^d} H_{G_f}(s, Y^{s^-},x) \left (K^{\hat{Y}}_s(dx) - K^{\hat{X}}_s(dx)\right ) \le 0.
\end{align*}
This process however is non-positive $dA \times P$ almost surely by Assumption~$(v)$ and using that $G_f$ is vertically increasing and vertically directional convex. Assumption~$(iv)$ then yields the proper supermartingale property.\\
If the inequalities in $(v)$ are reversed and $(G_f(t, Y^t)^+)_{t \in [0,T]}$ is of class (DL), $G_f$ is a submartingale.
\end{proof}

Next we transfer the comparison result to the case when $G_f$ is vertically convex and vertically increasing. 

\begin{theorem}[Vertically increasing and vertically convex comparison under P]
\label{thm:comppvcx}
Let $X,Y$ be special semimartingales and let $X_0 = x_0 = Y_0$ almost surely. Let $f \in L^1(P^{X^T}) \cap L^1(P^{Y^T})$. Assume that
\begin{itemize}
\item[(i)] $G_f \in C^{1,2}(\Lambda_T^d)$ and $G_f$ is vertically convex and vertically increasing on $\Lambda_T$;

\item[(ii)] -- $(iv)$ of Theorem \ref{thm:orderingpathidcxp} hold;

\item[(v)] The differential characteristics are $dA \times P$ almost surely ordered for all $i \le d$:
\begin{align*}
b_t^{\hat{Y} i} \le&~ b_t^{\hat{X} i},\\
c_t^{\hat{Y}} \le&_{psd}~ c_t^{\hat{X}},\\
\int_{\RR^d} H_{G_f}(t,Y^{t^-},x) K^{\hat{Y}}_t(dx) \le&~ \int_{\RR^d} H_{G_f}(t,Y^{t^-},x) K^{\hat{X}}_t(dx).
\end{align*}
\end{itemize} 

Then it holds that
\begin{align*}
E[f(Y^T)] \le E[f(X^T)].
\end{align*}
If in $(v)$ the inequalities are reversed and $(G_f(t, Y^t)^+)_{t \in [0,T]}$ is of class (DL), we get
\begin{align*}
E[f(Y^T)] \ge E[f(X^T)].
\end{align*}
\end{theorem}

\begin{proof}
Again using the functional It\^o formula we have to verify, that $dA \times P$ a.s.
\begin{align*}
\sum_{i \le d} \nabla_{\omega i} G_f(s, Y^{s^-}) \left(b^{\hat{Y} i}_s - b^{\hat{X}i}_s \right) + \frac{1}{2} \sum_{i,j \le d} \nabla^2_{\omega ij} G_f(s,Y^{s^-}) \left (c^{\hat{Y}ij}_s - c^{\hat{X}ij}_s\right )\\
+ \int_{\RR^d} H_{G_f}(s,Y^{s^-},x) \left (K^{\hat{Y}}_s(dx) - K^{\hat{X}}_s(dx)\right ) \le 0.
\end{align*}
The first term is non positive due to Assumption~$(v)$ and the fact that $G_f$ is vertically increasing in the second variable. The remaining part is non-positive as in the proof of Theorem \ref{thm:comparisonpathemmcx}. By Assumption~$(iv)$ it follows that $(G_f(t,Y^t))_{t \in[0,T]}$ is a proper supermartingale.\\
If the inequalities in $(v)$ are reversed and $(G_f(t, Y^t)^+)_{t \in [0,T]}$ is of class (DL), then $G_f$ is a submartingale.
\end{proof}

As before the key inequality of the proof can be used to formulate a comparison result without the assumption of vertical convexity and vertical monotonicity on the functional $G_f$.

\begin{corollary}[General comparison condition under P]
Let $X,Y$ be special semimartingales and let $X_0 = x_0 = Y_0$ almost surely. Let $f \in L^1 (P^{X^T}) \cap L^1(P^{Y^T})$. Assume that $G_f \in C_{loc}^{1,2}(\Lambda_T^d)$ and that $(ii)$--$(iv)$ of Theorem \ref{thm:orderingpathidcxp} hold. Further assume that $dA \times P$ almost surely
\begin{align}
\begin{split}
\label{eq:orderingoass2p}
&~\sum_{i \le d} \nabla_{\omega i} G_f(s, Y^{s^-}) \left(b^{\hat{Y} i}_s - b^{\hat{X}i}_s \right) + \frac{1}{2} \sum_{i,j \le d} \nabla^2_{\omega ij} G_f(s,Y^{s^-}) \left (c^{\hat{Y}ij}_s - c^{\hat{X}ij}_s\right ) \\
&~ + \int_{\RR^d} H_{G_f}(s,Y^{s^-},x) \left (K^{\hat{Y}}_s(dx) - K^{\hat{X}}_s(dx)\right ) \le 0.
\end{split}
\end{align}
Then it holds that
\begin{align*}
E[f(X^T)] \le E[f(Y^T)].
\end{align*}
If inequality \eqref{eq:orderingoass2p} is reversed and $(G_f(t, Y^t)^+)_{t \in [0,T]}$ is of class (DL), then
\begin{align*}
E\left [f(X^T)\right ] \ge E\left [f(Y^T)\right ].
\end{align*}
\end{corollary}

\begin{proof}
As in the proof of Theorem \ref{thm:comppvcx} inequality \eqref{eq:orderingoass2p} implies that $(G_f(t,Y^t))_{t \in [0,T]}$ is a supermartingale or submartingale respectively.
\end{proof}

\section{Results on regularity and applications}

The comparison results in Section \ref{sec:pathdepcomp} need various properties of the valuation functional $G_f$, like continuity, vertical/horizontal differentiability and convexity. In this section we give some results establishing these regularity properties and some applications to comparison results.

We first discuss the regularity of $G_f$. For notational simplicity we consider the processes under the semimartingale measure $P$. 

An example for a vertically differentiable conditional expectation is given in \citet[Proposition 4.4]{Ri15} who states conditions such that the conditional expectation of a path-dependent function of a semimartingale can be represented as horizontally differentiable non-anticipative functional. The underlying process is a stochastic exponential defined by the SDE
\begin{align*}
dS_t = S_t \sigma_t dB_t,
\end{align*}
where $B$ is a standard Brownian motion and $(\sigma_t)_{t \in [0,T]}$ is a non-negative adapted process such that $S$ is a $L^2$-martingale.

We modify this approach to transfer it to non-continuous processes. Therefore, we consider the probability space $(\Omega,(\cF_t)_{t \in [0,T]},\cF,P)$, where $\Omega = D([0,T],\RR^d)$, $\cF$ is the Borel sigma-field and $(\cF_t)_{t \in [0,T]}$ is the filtration generated by the canonical process, $X_t(\omega) = \omega(t)$. We assume that the canonical process is a semimartingale. 

In the center of our considerations in the previous section is the valuation functional $G_f: \Lambda^d_T \to \RR$,
\begin{align*}
G_f(t,\omega) := E\left [\left .f(X^T)\right |X^t = \omega^t\right ].
\end{align*}
In the setting of this section this is the same as the expectation w.r.t. factorized conditional probability of $X^T$ given $\cF_t$ due to the fact that $(\cF_t)_{t \in [0,T]}$ is the natural filtration. If the space of c\`{a}dl\`{a}g functions is equipped with the Skorokhod topology there exists a regular version of the conditional probability of $X^T$ given $X^t$ since $D([0,T],\RR^d)$ then is a Polish space. However, for the sup norm this is not valid anymore, see \citet{Bi68}. We assume in the sequel that a regular version of the conditional probability exists as in the case of processes with continuous paths. Then $G_f$ takes the form
\begin{align*}
G_f(t,\omega) = \int_{D([0,T],\RR^d)} f(\tilde{\omega}) P^{X^T|X^t=\omega^t}(d\tilde{\omega}).
\end{align*}
and, therefore, the horizontal and vertical differentiability is mainly a question of correspondent differentiability of the kernel $P^{X^T|X^t}$.

Since the metric in the space of stopped paths uses in the path component the sup norm, we need a tool to handle the sup norm of a semimartingale. This motivates the use of the class of $H^1$ semimartingales (for details see \citet{Pr05}). Without loss of generality we assume that all semimartingales in this section start in zero. For simplicity we consider one-dimensional semimartingales. Let $X$ be a semimartingale; then there exists at least one decomposition $X = M + B$, where $M$ is a local martingale and $B$ is of finite variation. Denoting for $1 \le p \le \infty$
\begin{align*}
j_p(M,B) := \left \| [M]_T^{\frac{1}{2}} + \int_0^T |dB_s| \right \|_{L^p},
\end{align*}
then the $H^p$ norm of $X$ is defined as
\begin{align*}
\| X \|_{H^p} = \inf_{X = M + B} j_p(M,B),
\end{align*}
where the infimum is taken over all possible semimartingale decompositions of $X$.

By \citet[Chapter V, Theorem 2]{Pr05} the $H^p$-norm allows to dominate the sup norm of $X$. This is a consequence of Burkholder's inequalities and is an important tool in the sequel. For $1 \le p \le \infty$ there exists a constant $c_p$ such that for any semimartingale $X$ with $X_0 = 0$ we have for $X^\ast := \sup_{t \in [0,T]} |X_t|$ the inequality
\begin{align}
\label{eq:domsupnorm}
\| X^\ast \|_{L^p} \le c_p \| X \|_{H^p}.
\end{align}
The following definition reminds the concatenation operators as introduced in \citet{Ri15}. In comparison we use a slightly different definition since we want the c\`{a}dl\`{a}g functions to meet in $t$.

\begin{definition}
The family of concatenation operators $(\oplus_t)_{t \in [0,T]}$ is defined by
\begin{align*}
\oplus_t : ~&D([0,T],\RR) \times D([0,T],\RR) \to D([0,T],\RR),\\
 & (\omega,\omega') \mapsto \omega \oplus_t \omega' := \omega \ind_{[0,t)} + (\omega_t + \omega' - \omega'_t) \ind_{[t,T]}.
\end{align*}
\end{definition}

The idea of the following theorem is to use Lipschitz continuity and independent increments to dominate the increments of the function under consideration. Then we are able to show the continuity and vertical and horizontal differentiability of $G_f$.

\begin{theorem}
\label{thm:examplepathdeplipschitz}
Let $X$ be a semimartingale with finite $H^1$ norm and independent increments without fixed times of discontinuity. Further, let $f: (D([0,T],\RR),\| \cdot \|_\infty) \to \RR$ be a Lipschitz continuous functional such that $E[|f(X^T)|] < \infty$. Assume that for any $\omega \in D([0,T],\RR)$ and any $t \in [0,T]$ the function 
\begin{align}
\begin{split}
\label{eq:hfromLipschitzexample}
g(\cdot,t,\omega):&~ \RR \to \RR\\
&~e ~\mapsto f(\omega + \ind_{[t,T]} e)
\end{split}
\end{align}
is twice continuously differentiable in zero such that the derivatives are Lipschitz continuous in $\omega$. Further, assume that for every $\omega, \omega' \in D([0,T],\RR)$ the function
\begin{align}
\begin{split}
\label{eq:ffromLipchitzexample}
l(\cdot, \omega, \omega'):&~ [0,T] \to \RR\\
&~t \mapsto f(\omega \oplus_t \omega')
\end{split}
\end{align}
is continuously right differentiable with derivative which is Lipschitz continuous in $\omega$. Then it follows that
\begin{align*}
G_f \in C^{1,2}_b(\Lambda_T).
\end{align*}
\end{theorem}

\begin{proof}
We obtain by the independence of increments that
\begin{align*}
G_f(t, \omega) = &~ E[f(X^T)| X^t = \omega^t]\\
= &~ E\left[f(\omega \oplus_t X^T) | X^t = \omega^t \right]\\
= &~ E\left[f\left(\omega \ind_{[0,t)} + (\omega_t + X^T - X_t) \ind_{[t,T]}\right)| \cF_t \right](\omega)\\
= &~ E\left[f(\omega \oplus_t X^T) \right].
\end{align*}
Since $f$ is Lipschitz continuous, it follows that for all $\omega, \omega' \in D([0,T], \RR)$ there exists a $c >0$ such that $|f(\omega) - f(\omega')| \le c \| \omega -\omega' \|_\infty$. We show the continuity of $G_f$ by sequential continuity. Let $((t^n,\omega^n))_{n \in \NN} \subset \Lambda_T$ converge to $(t,\omega), \in \Lambda_T$. Then we have
\begin{align}
\label{eq:jointcontinuityexample}
|G_f(t,\omega) - G_f(t^n,\omega^n)| = &~\big|E\left[ f(\omega \oplus_t X^T) \right] - E\left[ f(\omega^n \oplus_{t^n} X^T) \right]\big| \nonumber \\
\le &~ E\left[ \big|f(\omega \oplus_t X^T) - f(\omega^n \oplus_{t^n} X^T) \big| \right] \nonumber \\
\le &~ c E\left[ \| (\omega \oplus_t X^T) - (\omega^n \oplus_{t^n} X^T) \|_\infty \right] \\
\le &~ c E\left[ \|(\omega - \omega^n) \ind_{[0,t \wedge t^n)} \|_\infty + \| \left( \omega_t + X^T - X_t - \omega^n \right) \ind_{[t, t^n)} \|_\infty \right. \nonumber \\
&~ + \| \left( \omega^n_{t^n} + X^T - X_{t^n} - \omega \right) \ind_{[t^n, t)} \|_\infty \nonumber\\
&~ \left. + \| \left( \omega_t + X^T - X_t - \omega^n_{t^n} - X^T + X_{t^n} \right) \ind_{[t \vee t^n,T]} \|_\infty \right]. \nonumber
\end{align}
This can be further dominated by
\begin{align*}
 c E\left[ \right .&\left .\|(\omega - \omega^n) \ind_{[0,t \wedge t^n)} \|_\infty + \| ( \omega_t - \omega^n) \ind_{[t, t^n)} \|_\infty \right. \\
&~ + \| ( X^T - X_t ) \ind_{[t, t^n)} \|_\infty + \| (\omega^n_{t^n} - \omega) \ind_{[t^n, t)} \|_\infty  \\
&~ + \| (X^T - X_{t^n}) \ind_{[t^n, t)} \|_\infty + \| ( \omega_t - \omega^n_{t^n}) \ind_{[t \vee t^n,T]} \|_\infty  \\
&~ \left.+ \| (X_{t^n} - X_t) \ind_{[t \vee t^n,T]} \|_\infty \right].
\end{align*} 
Note that for fix $n$ only one of the indicator functions $\ind_{[t^n, t)}$ and $\ind_{[t, t^n)}$ differs from zero. We consider the first term on the right-hand side. It is clearly bounded from above by $\| \omega^t - (\omega^n)^{t^n} \|_\infty $ which tends to zero by $d_\infty$ convergence. As consequence we obtain
\begin{align*}
E\left[ \|(\omega - \omega^n) \ind_{[0,t \wedge t^n)} \|_\infty \right] \le E\left[ \| \omega^t - (\omega^n)^{t^n} \|_\infty \right] = \| \omega^t - (\omega^n)^{t^n} \|_\infty \to 0.
\end{align*}
The same argument yields convergence to zero for the other terms containing $\omega$ and $\omega^n$.

Next we consider the expectation $E \left[\| ( X^T - X_t) \ind_{[t, t^n)} \|_\infty \right]$. The process therein $X^n :=(( X^T - X_t) \ind_{[t, t^n)})_{t \in [0,T]}$ is a semimartingale starting in zero, hence we can apply \eqref{eq:domsupnorm} with $p = 1$ to obtain
\begin{align*}
E \left[ \| ( X^T - X_t) \ind_{[t, t^n)} \|_\infty \right] \le c_1 \| X^n \|_{H^1}.
\end{align*}
Let $X = M + A$ be a semimartingale decomposition of $X$. Then after a restriction to $M^n :=(( M^T - M_t) \ind_{[t, t^n)} )_{t \in [0,T]}$ and $A^n :=(( A^T - A_t) \ind_{[t, t^n)})_{t \in [0,T]}$ we get that $X^n = M^n + A^n$ is a semimartingale decomposition of $X^n$. Since for each $n$ and $\omega\in \Omega$ the path $X^n(\omega)$ is just a shifted piece of the path of $X(\omega)$, we have that $[M^n]_t \le [M]_t$ and $|A^n_t| \le |A_t|$ for all $t \in [0,T]$. This means that we can dominate the $H^1$ norm of all $X^n$ by the $H^1$ norm of $X$ which is finite by assumption. Dominated convergence and right continuity then leads to
\begin{align*}
\lim_{n \to \infty} E \left[ \| (X^T - X_t) \ind_{[t,t^n)} \|_\infty \right] = E \left[ \lim_{n \to \infty} \sup_{s \in [t,t^n)} |X_s - X_t| \right] = 0.
\end{align*}
Analogously we get that
\begin{align*}
\lim_{n \to \infty} E \left[ \| (X^T - X_{t^n}) \ind_{[t^n,t)} \|_\infty \right] =&~ E \left[ \lim_{n \to \infty} \sup_{s \in [t^n,t)} |X_s - X_{t^n}| \right]\\
\le&~ E \left[ \lim_{n \to \infty} \sup_{s \in [t^n,t]} |X_s - X_{t^n}| \right]\\
=&~ E \left[ |\Delta X_t| \right]\\
=&~ 0.
\end{align*}
The last equality follows from the assumption that there are no fixed times of discontinuity. It remains to show that $E[\| (X_{t^n} - X_t) \ind_{[t \vee t^n,T]} \|_\infty]$ also tends to zero. Therefore, note that
{\small
\begin{align*}
E\left[\| (X_{t^n} - X_t) \ind_{[t \vee t^n,T]} \|_\infty \right] = E\left[ |(X_{t^n} - X_t )| \right] \le E \left[ \sup_{s \in [t^n,t]} |X_s - X_{t^n}| + \sup_{s \in [t,t^n]} |X_s - X_t|\right] .
\end{align*}}
The terms on the right-hand side are both bounded by the $H^1$ norm of $X$. It follows by dominated convergence that this tends to zero. Thus, $G_f$ is continuous.

Next we show that $G_f$ is vertically differentiable. We consider the vertical difference quotient of $G_f$
\begin{align*}
\frac{G_f(t,\omega^{h,t}) - G_f(t,\omega)}{h} = &~ \frac{1}{h} \left( E\left[f(\omega^{h,t} \oplus_t X^T) \right] - E\left[f(\omega \oplus_t X^T) \right] \right)\\
= &~ \frac{1}{h} E\left[f(\omega \oplus_t X^T + h \ind_{[t,T]}) - f(\omega \oplus_t X^T) \right].
\end{align*}
Since $f$ is Lipschitz continuous, dominated convergence yields
\begin{align*}
\nabla_\omega G_f(t,\omega) = E\left [\frac{\partial}{\partial e} g(e,t,\omega \oplus_t X^T)(0)\right ].
\end{align*}
For the second derivative we use that the first derivative of $g$ is assumed to be Lipschitz continuous in $\omega$ and get by dominated convergence 
\begin{align*}
\nabla^2_\omega G_f(t,\omega) = &~ \lim_{h \to 0} \frac{\nabla_\omega G_f(t,\omega^{h,t}) - \nabla_\omega G_f(t,\omega)}{h}\\
= &~ E\left[ \lim_{h \to 0} \frac{\frac{\partial}{\partial e} g(e,t,\omega^{h,t} \oplus_t X^T)(0) - \frac{\partial}{\partial e} g(e,t,\omega \oplus_t X^T)(0)}{h}\right]\\
=&~ E\left[ \frac{\partial^2}{\partial e^2} g(e,t,\omega \oplus_t X^T)(0) \right].
\end{align*}
We are left to show that $\nabla_\omega G_f$ and $\nabla^2_\omega G_f$ are (left-)continuous. In fact we have continuity which follows as the continuity of $G_f$ from Lipschitz continuity.

We now turn to the horizontal differentiability. Therefore, we consider the horizontal difference quotient.
\begin{align*}
\frac{G_f(t+h,\omega^t) - G_f(t,\omega^t)}{h} = \frac{E\left[ f(\omega^t \oplus_{t+h} X^T) \right] - E\left[ f(\omega^t \oplus_t X^T) \right]}{h}.
\end{align*} 
From the Lipschitz continuity of $f$ it follows as in \eqref{eq:jointcontinuityexample} that the difference is bounded by the $H^1$ norm of $X$. With dominated convergence it follows for $h \downarrow 0$ that
\begin{align*}
\cD G_f(t,\omega) = E \left[ \frac{\partial^+}{\partial t} l(t,\omega^t,X^T) \right].
\end{align*}
The continuity of the derivative now follows from the Lipschitz continuity of the derivative of $l$.

It remains to show that $G_f$ is boundedness preserving. Therefore, let be $K \subset \RR$ be compact and $t_0$ fixed. We need to show the existence of a constant $C_{K,t_0} > 0$ such that for all $t \le t_0$ and all $\omega \in D([0,T],\RR)$ we have
\begin{align*}
\omega([0,t]) \subset K \Rightarrow |G_f(t,\omega)| \le C_{K,t_0}.
\end{align*}
Since $K$ is compact it is bounded; let $k$ be this bound. We obtain from \eqref{eq:jointcontinuityexample} and the considerations thereafter that 
\begin{align*}
|G_f(t,\omega) - G_f(0,0)| \le c E[2 \| \omega^t \|_\infty + 2 \| X \|_{H^1}] \le c (2k + 2 \| X \|_{H^1}) =: \tilde{C}.
\end{align*}
The term $G_f(0,0)$ is just $E[f(X^T)]$ which is finite. So we get by the choice $C_{K,t_0} = \tilde{C} + |E[f(X^T)]|$ that $G_f$ is boundedness preserving.
\end{proof}

\begin{remark}
\label{rem:exampleLipschitzpath}
\begin{enumerate}
\item The Lipschitz continuity helps to show continuity and to apply dominated convergence. H\"{o}lder continuity as condition on the functions above works as well.

\item The property to be boundedness preserving is a local property; it depends on $t_0$. In the proof we have seen that under the conditions of Theorem \ref{thm:examplepathdeplipschitz} $G_f$ is even ``globally'' boundedness preserving.

\item By \citet[Corollary II.4.18]{JS03} the property ``without fixed times of continuity'' is for processes with independent increments equivalent to quasi-left-continuity of $X$.

\item The functions $g$ and $l$ from equations \eqref{eq:hfromLipschitzexample} and \eqref{eq:ffromLipchitzexample} provide the vertical and horizontal differentiability. If only one of the functions has the demanded properties, we still get $G_f \in C_b^{0,2}(\Lambda_T)$ or $G_f \in C_b^{1,0}(\Lambda_T)$.
\end{enumerate}
\end{remark}
 
We give an example for a semimartingale and the integral functional from Example \ref{ex:exampleregularintroductfuncito} which fulfill the conditions of Theorem \ref{thm:examplepathdeplipschitz}.

\begin{example}
\label{ex:integraloverLipschitzfunction}
Let $X$ be a compound Poisson process with finite $H^1$ norm. Then it has no fixed times of discontinuity, see \citet[II.4.3]{JS03}. Further, let $\tilde{f}: \RR \to \RR$ be Lipschitz continuous and twice continuously differentiable with Lipschitz continuous derivatives and let $f$ be the integral $f(\omega) := \int_0^T \tilde{f}(\omega_t) dt$. Assume that $E[|f(X^T)|] < \infty$.

Then $f$ is Lipschitz continuous in $\omega \in D([0,T],\RR)$. This is consequence of the Lipschitz continuity of $\tilde{f}$:
\begin{align*}
|f(\omega) - f(\omega')| \le &~ \int_0^T |\tilde{f}(\omega_t) - \tilde{f}(\omega'_t)| dt\\
\le &~ c \int_0^T |\omega_t - \omega'_t| dt\\
\le &~ c T \| \omega - \omega'\|_\infty.
\end{align*} 
Further, the function $g$ from equation \eqref{eq:hfromLipschitzexample} is twice continuously differentiable in zero. To see this fix $s \in [0,T]$ and $\omega \in D([0,T],\RR)$; then we have by Lipschitz continuity and dominated convergence
\begin{align*}
\frac{\partial}{\partial e} g(e,s,\omega)(0) =&~ \lim_{h \to 0} \frac{f(\omega + \ind_{[s,T]}h) - f(\omega)}{h}\\
= &~ \lim_{h \to 0} \frac{\int_0^T \tilde{f}(\omega_t + \ind_{[s,T]}h) - \tilde{f}(\omega_t) dt}{h}\\
=&~ \lim_{h \to 0} \frac{\int_s^T \tilde{f}(\omega_t + h) - \tilde{f}(\omega_t) dt}{h}\\
=&~ \int_s^T \tilde{f}'(\omega_t) dt.
\end{align*}
This expression is Lipschitz continuous in $\omega$ since we assumed $\tilde{f}'$ to be Lipschitz continuous.\\
Analoguously we get 
\begin{align*}
\frac{\partial^2}{\partial e^2} g(e,s,\omega)(0) = \int_s^T \tilde{f}''(\omega_t) dt,
\end{align*}
which is Lipschitz continuous in $\omega$ as well. Thus, $g$ fulfills the conditions of Theorem \ref{thm:examplepathdeplipschitz}.

For the function $l$ from \eqref{eq:ffromLipchitzexample} we show now the right differentiability. Therefore, fix $\omega,\omega' \in D([0,T],\RR)$, then
\begin{align*}
\frac{\partial^+}{\partial t} l(t,\omega,\omega') = &~ \lim_{h \downarrow 0} \frac{f(\omega \oplus_{t+h} \omega') - f(\omega \oplus_t \omega')}{h}\\
=&~ \lim_{h \downarrow 0} \frac{1}{h} \left (\int_0^{t+h} \tilde{f}(\omega_s) ds + \int_{t+h}^T \tilde{f}(\omega'_s - \omega'_{t+h} + \omega_t) ds - \int_0^t \tilde{f}(\omega_s) ds \right .\\
&~ \left . - \int_t^T \tilde{f}(\omega'_s - \omega'_t + \omega_t) ds \right )\\
=&~ \lim_{h \downarrow 0} \frac{1}{h} \left (\int_t^{t+h} \tilde{f}(\omega_s) - \tilde{f}(\omega'_s - \omega'_t + \omega_t) ds   \right .\\
&~ \left . + \int_{t+h}^T \tilde{f}(\omega'_s - \omega'_{t+h} + \omega_t) - \tilde{f}(\omega'_s - \omega'_t + \omega_t) ds\right )\\
=&~ \lim_{h \downarrow 0} \frac{1}{h} \left (\int_0^T \tilde{f}(\omega'_s - \omega'_{t+h} + \omega_t) - \tilde{f}(\omega'_s - \omega'_t + \omega_t) ds\right . \\
&~ - \int_0^{t+h} \tilde{f}(\omega'_s - \omega'_{t+h} + \omega_t) ds + \int_t^{t+h} \tilde{f}(\omega_s) - f(\omega'_s - \omega'_t + \omega_t) ds\\
&~ \left . + \int_0^t \tilde{f}(\omega'_s - \omega'_t + \omega_t) ds + \int_t^{t+h} \tilde{f}(\omega'_s - \omega'_t + \omega_t) ds \right )\\
=&~ \int_0^T \tilde{f}'(\omega'_s - \omega'_t + \omega_t) \frac{\partial^+}{\partial t} \omega'_t ds - \tilde{f}(\omega_t) \tilde{f}'(\omega_t) \frac{\partial^+}{\partial t} \omega'_t + \tilde{f}(\omega_t).
\end{align*}
The first term results from dominated convergence, the second term is the right derivative of the integral $\int_0^u \tilde{f}(\omega'_s - \omega'_u + \omega_t) ds$. For a compound Poisson process, the path $\omega' = X$ is right differentiable and it follows that on such paths $\cD G_f = \tilde{f}(\omega_t)$.\\
That $G_f$ is boundedness preserving follows as in the proof of Theorem \ref{thm:examplepathdeplipschitz}. Altogether we have that $G_f \in C_b^{1,2}(\Lambda_T)$.
\end{example}

From this example one can see that in this setting the function $l$ can cause problems for more general semimartingales since in the derivation a right derivative of the future path occurred. In fact this proceeding works fine for semimartingales of finite variation since they are differentiable almost everywhere. But since integrals over path independent functions of semimartingales are not of finite variation, we need other conditions for horizontal differentiability.

\begin{example}
\label{ex:examplebrownianmotionpath}
Let $B$ be a Brownian motion. We consider the function $f(\omega) := \int_0^T \tilde{f}(\omega_t) dt$ from Example \ref{ex:integraloverLipschitzfunction}. In contrast to the previous example we only assume that $\tilde{f}$ is bounded. Then we have by the Markov property and the strong continuity of the corresponding transition semigroup $(T_t)_{0 \le t \le T}$ that the transition semigroup is differentiable in time. It follows that $G_f$ is horizontally differentiable:
\begin{align*}
\cD G_f(t,\omega) = &~ \lim_{h \downarrow 0} \frac{E[f(\omega^t \oplus_{t+h} B^T) - f(\omega^t \oplus_t B^T)]}{h}\\
=&~ \lim_{h \downarrow 0} \frac{E[\int_0^T \tilde{f}((\omega^t \oplus_{t+h} B^T)_s) - \tilde{f}((\omega^t \oplus_t B^T)_s) ds]}{h}\\
=&~ \lim_{h \downarrow 0} \frac{\int_t^T E[\tilde{f}((\omega^t \oplus_{t+h} B^T)_s) - \tilde{f}((\omega^t \oplus_t B^T)_s)] ds}{h}\\
=&~ \lim_{h \downarrow 0} \frac{\int_t^T T_{(s-h)\wedge0}\tilde{f}(\omega_t) - T_s \tilde{f}(\omega_t) ds}{h}\\
=&~ \int_t^T \frac{\partial}{\partial s} T_s \tilde{f}(\omega_t) ds.
\end{align*}
\end{example}

The most important part in the proof of Theorem \ref{thm:examplepathdeplipschitz} is that we are able to reduce the conditional expectation to a normal expectation. This is a consequence of the independent increments. Since Markov processes have conditionally independent increments, we can obtain a similar result. In fact in the following proposition we derive for a Feller semimartingale $X$ under the assumption that $f$ is an integral function that $G_f \in C^{1,2}_b(\Lambda_T^d)$.

We recall the notion of Feller processes, cf. \citet{EK05}. A semigroup $(T_t)_{t \ge 0}$ on $C_0(\RR^d)$ is called a Feller semigroup if it is strongly continuous. In particular, Feller semigroups map $C_0(\RR^d)$ to $C_0(\RR^d)$. Sometimes also the bigger class $C_b(\RR^d)$ is used in the definition of Feller semigroups, we denote this by $C_b$-Feller semigroup. A Markov process is called a Feller process if the corresponding transition semigroup is a Feller semigroup. If $X$ is in addition a semimartingale we call it a Feller semimartingale. Note that by this definition Feller processes are always time-homogeneous.\\
Examples of $C_b$-Feller processes are L\'{e}vy processes. For L\'evy processes it holds that the transition semigroup is of the form
\begin{align*}
T_t f(x) = \int_\RR f(y+x) p_t(dy).
\end{align*}
Here $p_t$ is the distribution of $X_t$. From this equation we see that $T_t f$ inherits the boundedness and continuity of $f$.

\begin{proposition}
\label{prop:fellerpath}
Let $X$ be a $C_b$-Feller semimartingale with strongly continuous transition semigroup $(T_t)_{0 \le t \le T}$. Further, let $\tilde{f}: \RR \to \RR$ be bounded and continuous such that $T_t \tilde{f} \in C^{1,2}$. We consider the function $f: (D([0,T],\RR),\| \cdot \|_{sup}) \to \RR$ to be the integral functional $f(\omega) = \int_0^T \tilde{f}(\omega_t)dt$.
Then it holds that $G_f \in C^{1,2}_b(\Lambda_T)$.
\end{proposition}

\begin{proof}
By the time-homogeneity and the Markov property we obtain
\begin{align}
\begin{split}
\label{eq:examplepathdepfeller}
G_f(t,\omega) =&~ E[f(X^T)|X^t = \omega^t]\\ 
=&~ E\left[ \left . \int_t^T \tilde{f}(X_s) ds \right | X_t = \omega_t \right] + \int_0^t \tilde{f}(\omega_s) ds\\
=&~ E\left[ \left . \int_0^{T-t} \tilde{f}(X_s) ds \right | X_0 = \omega_t \right] + \int_0^t \tilde{f}(\omega_s) ds\\
=&~ \int_0^{T-t} T_s \tilde{f}(\omega_t) ds + \int_0^t \tilde{f}(\omega_s) ds.
\end{split}
\end{align}
We show the continuity of $G_f$ by sequential continuity. Let $(t^n,\omega^n)$ converge in $\Lambda_T$ to $(t,\omega)$. Then we have 
\begin{align*}
\MoveEqLeft |G_f(t,\omega) - G_f(t^n,\omega^n)|\\
= &~ \left |E\left[ \left .\int_0^{T-t} \tilde{f}(X_s) ds\right | X_0 = \omega_t \right] - E\left[ \left .\int_0^{T-t^n} \tilde{f}(X_s) ds\right | X_0 = \omega^n_{t^n} \right] \right .\\
&~ \left .+ \int_0^t \tilde{f}(\omega_s) ds - \int_0^{t^n} \tilde{f}(\omega^n_s) ds \right |\\
=&~ \left |\int_0^{T-(t \vee t^n)} E\left[ \left .\tilde{f}(X_s)\right | X_0 = \omega_t \right] - E\left[\left . \tilde{f}(X_s) \right | X_0 = \omega^n_{t^n} \right] ds \right .\\
&~ + \int_{T - (t \vee t^n)}^{T -(t \wedge t^n)} E\left[ \left .\tilde{f}(X_s)\right | X_0 = \omega_t \right] \ind_{\{t \ge t^n\}} - E\left[ \left .\tilde{f}(X_s) \right | X_0 = \omega^n_{t^n} \right] \ind_{\{t^n \ge t\}} ds\\
&~\left . + \int_0^{t \wedge t^n} \tilde{f}(\omega_s) - \tilde{f}(\omega^n_s) ds + \int_{t \wedge t^n}^{t \vee t^n} \tilde{f}(\omega_s)\ind_{\{t \ge t^n\}} - \tilde{f}(\omega^n_s)\ind_{\{t^n \ge t\}} ds\right |
\end{align*}
We first take a closer look at the integrals not depending on $X$. 
\begin{align*}
\left |\int_0^{t \wedge t^n} \tilde{f}(\omega_s) - \tilde{f}(\omega^n_s) ds \right | \le &~ \int_0^{t \wedge t^n} |\tilde{f}(\omega_s) - \tilde{f}(\omega^n_s)| ds\\
\le& ~ \int_0^t |\tilde{f}(\omega_s) - \tilde{f}(\omega^n_s)| ds.
\end{align*}
This converges to zero by dominated convergence using the continuity of $\tilde{f}$. Let $c$ be the bound of $\tilde{f}$, then we have
\begin{align*}
\left |\int_{t \wedge t^n}^{t \vee t^n} \tilde{f}(\omega_s)\ind_{\{t \ge t^n\}} - \tilde{f}(\omega^n_s)\ind_{\{t^n \ge t\}} ds\right | \le &~ \int_{t \wedge t^n}^{t \vee t^n} \left |\tilde{f}(\omega_s)\ind_{\{t \ge t^n\}} - \tilde{f}(\omega^n_s)\ind_{\{t^n \ge t\}}\right | ds\\
\le &~ c (t \vee t^n - t \wedge t^n).
\end{align*}
This tends to zero by assumption. Next we turn to the terms containing $X$. For the first term we obtain
\begin{align*}
\left |\int_0^{T-(t \vee t^n)} E\left[ \left .\tilde{f}(X_s)\right | X_0 = \omega_t \right] - E\left[ \left .\tilde{f}(X_s) \right | X_0 = \omega^n_{t^n} \right] ds \right |\\
\le \int_0^{T-t} \left | E\left[ \left .\tilde{f}(X_s)\right | X_0 = \omega_t \right] - E\left[ \left.\tilde{f}(X_s) \right | X_0 = \omega^n_{t^n} \right]\right | ds
\end{align*}
This converges to zero since $E\left[ \left .\tilde{f}(X_s)\right | X_0 = \omega_t \right]$ is continuous in $\omega$ and bounded by the Feller property.

The last term tends to zero as follows. Let $\tilde{c}$ be the bound of $E\left[ \left .\tilde{f}(X_s)\right | X_0 = \omega_t \right]$ which exists by the Feller property. Then
\begin{align*}
\left |\int_{T - (t \vee t^n)}^{T -(t \wedge t^n)} E\left[ \left .\tilde{f}(X_s)\right | X_0 = \omega_t \right] \ind_{\{t \ge t^n\}} - E\left[ \left .\tilde{f}(X_s) \right | X_0 = \omega^n_{t^n} \right] \ind_{\{t^n \ge t\}} ds\right |\\
\le \tilde{c} (T - (t \vee t^n) - T +(t \wedge t^n)) \to 0.
\end{align*}
We now turn to the vertical differentiability of $G_f$. 
\begin{align*}
\frac{G_f(t,\omega^{h,t}) - G_f(t,\omega)}{h} = &~ \frac{1}{h} \left ( E\left[ \left .\int_0^{T-t} \tilde{f}(X_s) ds\right | X_0 = \omega_t + h \right] + \int_0^t \tilde{f}(\omega_s) ds \right .\\
&~ \left . - E\left[ \left .\int_0^{T-t} \tilde{f}(X_s) ds\right | X_0 = \omega_t \right] - \int_0^t \tilde{f}(\omega_s) ds \right )\\
=&~ \frac{1}{h} \left ( \int_0^{T-t} E\left[ \left .\tilde{f}(X_s) \right | X_0 = \omega_t + h \right] - E \left [ \left .\tilde{f}(X_s) \right | X_0 = \omega_t \right] ds \right )\\
=&~ \frac{1}{h} \left ( \int_0^{T-t} T_s \tilde{f}(\omega_t + h) - T_s \tilde{f}(\omega_t) ds \right )
\end{align*}
Since $T_t \tilde{f} \in C^{1,2}$ by assumption, we obtain
\begin{align*}
\nabla_{\omega} G_f(t,\omega) = \int_0^{(T-t)} \frac{\partial}{\partial x} T_s \tilde{f}(\omega_t)ds.
\end{align*}
Analog we receive for the second derivative
\begin{align*}
\nabla^2_{\omega} G_f(t,\omega) = \int_0^{(T-t)} \frac{\partial^2}{\partial x^2} T_s \tilde{f}(\omega_t)ds.
\end{align*}
To compute the horizontal derivative we take a look at the horizontal differential quotient.
\begin{align*}
\frac{G_f(t+h,\omega^t) - G_f(t,\omega)}{h} =&~ \frac{1}{h} \left ( E\left[ \int_0^{T-t-h} \tilde{f}(X_s) ds| X_0 = \omega_t \right] + \int_0^t \tilde{f}(\omega_s) ds + h \tilde{f}(\omega_t)\right .\\
&~ \left . - E\left[ \int_0^{T-t} \tilde{f}(X_s) ds| X_0 = \omega_t \right] - \int_0^t \tilde{f}(\omega_s) ds \right )\\
=&~ \tilde{f}(\omega_t) - \frac{1}{h} \int_{T-t-h}^{T-t} T_s \tilde{f}(\omega_t) ds\\
\to&~ \tilde{f}(\omega_t) - T_{T-t} \tilde{f}(\omega_t).
\end{align*} 
It remains to show that $G_f$ is boundedness preserving. This follows directly from the representation \eqref{eq:examplepathdepfeller} since $\tilde{f}$ and $T_t \tilde{f}$ are both bounded by the Feller property.
\end{proof}

Next we consider functions which depend of the average of a semimartingale $X$. Such functions are used in financial mathematics in the framework of Asian options.

\begin{example}
\label{ex:exampleasian}
Let $X$ be a semimartingale of finite variation with independent increments, finite $H^1$ norm and without fixed times of discontinuity. Further, define $I_t := \int_0^t X_s ds$. We consider a function of the form $\tilde{f}(\frac{1}{T} I_T)$, where $\tilde{f} : \RR \to \RR$ is an integrable function. We assume that $\tilde{f}$ is twice differentiable with Lipschitz continuous derivatives. The path-dependent function corresponding to $\tilde{f}$ is
\begin{align*}
f : D([0,T],\RR) &\to \RR\\
\omega~~~ &\mapsto \tilde{f} \left(\frac{1}{T} \int_0^T \omega_t dt \right).
\end{align*} 
Since the identity on $\RR$ is Lipschitz continuous, we get as in Example \ref{ex:integraloverLipschitzfunction} that $\frac{1}{T} I_T$ is Lipschitz continuous in $\omega$. It follows that $f$ is Lipschitz continuous in $\omega$ as well. Denote by $c_{\tilde{f}}$ the Lipschitz constant of $\tilde{f}$ and by $c_I$ the Lipschitz constant of $\frac{1}{T} I_T$. Then it follows
\begin{align*}
|f(\omega) - f(\omega')| =&~ \left |\tilde{f} \left(\frac{1}{T} \int_0^T \omega_t dt \right) - \tilde{f} \left(\frac{1}{T} \int_0^T \omega'_t dt \right) \right | \\
\le&~ c_{\tilde{f}} \left |\frac{1}{T} \int_0^T \omega_t dt - \frac{1}{T} \int_0^T \omega'_t dt \right | \le c_{\tilde{f}} c_I \| \omega - \omega' \|.
\end{align*}
We consider the function $g$ from equation \eqref{eq:hfromLipschitzexample}. Fix $s \in [0,T]$ and $\omega \in D([0,T],\RR)$, then we get for the derivative

\begin{align*}
\frac{\partial}{\partial e} g(e,s,\omega)(0) =&~ \lim_{h \to 0} \frac{f(\omega + \ind_{[s,T]}h) - f(\omega)}{h}\\
=&~ \lim_{h \to 0} \frac{\tilde{f} \left(\frac{1}{T} \int_0^T \omega_t dt + \frac{T-s}{T} h \right) - \tilde{f} \left(\frac{1}{T} \int_0^T \omega_t dt \right)}{h}\\
=&~ \frac{T-s}{T} \tilde{f}'\left(\frac{1}{T} \int_0^T \omega_t dt \right).
\end{align*}

This is Lipschitz continuous in $\omega$ by the same argument as above. For the second derivative we get

\begin{align*}
\frac{\partial^2}{\partial e^2} g(e,s,\omega)(0) = \frac{(T-s)^2}{T^2}\tilde{f}'' \left( \frac{1}{T} \int_0^T \omega_t dt \right ).
\end{align*}

So $g$ meets the conditions of Theorem \ref{thm:examplepathdeplipschitz}. For the horizontal derivative, we get

{\small
\begin{align*}
\MoveEqLeft \cD G_f(t,\omega)\\ 
=&~ \lim_{h \downarrow 0} \frac{E\left [f(\omega^t \oplus_{t+h} X^T) - f(\omega^t \oplus_t X^T)\right ]}{h}\\
=&~ \lim_{h \downarrow 0} \frac{1}{h} E\left [\tilde{f}\left (\frac{1}{T} \left (\int_0^t \omega_s ds + h \omega_t + \int_{t+h}^T X_s - X_{t+h} + \omega_t ds \right ) \right )\right .\\
&~~~~~~~~~~~~~~ \left .- \tilde{f} \left ( \frac{1}{T} \left ( \int_0^t \omega_s ds + \int_t^T X_s - X_t + \omega_t ds\right ) \right )\right ]\\
=&~ \lim_{h \downarrow 0} \frac{1}{h} E\left [\tilde{f}\left (\frac{1}{T} \left (\int_0^t \omega_s ds + \int_t^T X_s + \omega_t ds - (T-t-h) X_{t+h} - \int_t^{t+h} X_s ds \right ) \right )\right .\\
&~~~~~~~~~~~~~~ \left .- \tilde{f} \left ( \frac{1}{T} \left ( \int_0^t \omega_s ds + \int_t^T X_s + \omega_t ds - (T-t) X_t \right ) \right )\right ].
\end{align*}}

We set $\int_0^t \omega_s ds + \int_t^T X_s + \omega_t ds := x$ and obtain by dominated convergence

\begin{align*}
\MoveEqLeft \cD G_f(t,\omega)\\
 =&~ E\left [\lim_{h \downarrow 0} \frac{1}{h}\tilde{f}\left (\frac{1}{T} \left (x - (T-t-h) X_{t+h} - \int_t^{t+h} X_s ds \right ) \right )\right .\\
&~~~~~~ \left .- \tilde{f} \left ( \frac{1}{T} \left ( x - (T-t) X_t \right ) \right )\right ].
\end{align*}

This can be further computed as

\begin{align*}
\MoveEqLeft \cD G_f(t,\omega)\\ 
=&~ E\left [\lim_{h \downarrow 0} \frac{1}{h}\tilde{f}\left (\frac{1}{T} \left (x - (T-t) X_t - (T-t)(X_{t+h} - X_t) + h X_{t+h} - \int_t^{t+h} X_s ds \right ) \right )\right .\\
&~~~~~~ \left .- \tilde{f} \left ( \frac{1}{T} \left ( x - (T-t) X_t \right ) \right )\right ]\\
=&~ E\left [\tilde{f}'(x - (T-t) X_t) T \frac{\partial^+}{\partial t} X_t \right ].
\end{align*}
\end{example}

Again we see that in the horizontal derivative a right derivative of the path occurs. This is the reason to restrict to finite variation semimartingales. In this case the Markov property does not help since for a Markov process $X$ the functional $G_f$ reduces to
\begin{align*}
G_f(t,\omega) =&~ E\left [\left . \tilde{f}\left (\frac{1}{T} I_T\right ) \right | X^t = \omega^t \right ]\\
=&~ E\left [ \left .\tilde{f} \left (\frac{1}{T} \left (\int_0^t \omega_s ds + \int_t^T X_s ds\right )\right ) \right | X_t = \omega_t  \right ]. 
\end{align*}
This representation does not allow to use the transition operators since the function in the conditional expectation depends on the whole path after $t$, whereas the transition operators only depend on the process at $t$.

The comparison results in Section \ref{sec:pathdepcomp} also need vertical convexity, vertical directional convexity and vertical monotonicity. In particular independent increments are useful to establish these properties as shown in the following example.

\begin{example}
\label{ex:exampleG_fverticallyconvex}
Let $X$ be a semimartingale with independent increments. Then $G_f$ is of the form $G_f(t,\omega) = E[f(\omega \oplus_t X^T)]$ (see Theorem \ref{thm:examplepathdeplipschitz}). To establish vertical convexity, we need to show that $G_f(t,\omega + e \ind_{[t,T]})$ is convex as function in $e$ in a neighbourhood of $0$. For the functional $f(\omega) := \int_0^T \tilde{f}(\omega_t) dt$ as in Example \ref{ex:integraloverLipschitzfunction}, we obtain
\begin{align*}
G_f(t, \omega + e \ind_{[t,T]}) =&~ E\left [f(\omega \oplus_t X^T + e \ind_{[t,T]})\right ] = \int_0^t \tilde{f}(\omega_s) ds + E\left[\int_t^T \tilde{f}(X_s +e) ds\right ].
\end{align*}
Thus, if $\tilde{f}$ is convex or $f$ is vertically convex, we obtain that $G_f$ is vertically convex. Analog statements hold for directional convexity and monotonicity.
\end{example}

We finally apply the regularity results in this section to obtain a comparison result for a path-dependent function between a L\'evy process and an It\^o semimartingale. Concretely the following example is based on Theorem \ref{thm:orderingpathidcxp}.

\begin{example}
Let $X$ be a type C L\'{e}vy process (see \citet{Sa99}) with L\'evy triplet $(b,c^2,K)$ and let $\tilde{f}: \RR \to \RR$ be a bounded, continuous, increasing directionally convex function. Then we have by Proposition \ref{prop:fellerpath} that for $f(\omega) = \int_0^T \tilde{f}(\omega_t)dt$ the functional $G_f$ is in $C^{1,2}_b(\Lambda_T)$. Further, $G_f$ is vertically directionally convex and vertically increasing by Example \ref{ex:exampleG_fverticallyconvex}. So condition $(i)$ of Theorem \ref{thm:orderingpathidcxp} is fulfilled.\\
We compare $X$ to an It\^o semimartingale $Y$ with differential characteristics $(\beta,\delta^2,\eta)$. Here $\beta$ is an adapted process which is integrable with respect to the identity, $\delta$ is an adapted process which is integrable with respect to the Brownian motion and $\eta$ is such that $\nu(dt,dx) := dt \eta_t(dx)$ is the compensator of $\mu^Y$. \\
Since $X$ is a type C L\'evy process we have that $\supp(P^{X_t}) = \RR$ for all $t$. Hence, by the choice of $f$ we have that for all $\omega \in \RR^{[0,T]}$
\begin{align*}
\bar{U}_t G_f(t,\omega^t) = 0.
\end{align*}
It follows that the generalized Kolmogorov backwards equation $\bar{U}_t G_f(t,Y^{t^-}) = 0$ holds for the path of $Y$ and consequently condition $(ii)$ is fulfilled.\\
If Assumptions~$(iii)$ and $(iv)$ are imposed, we get from the $dt \times P$ almost sure ordering of the differential characteristics
\begin{align*}
\beta_t \le&~ b,\\
\delta_t \le&~ c,\\
\int_\RR H_{G_f}(t,Y^{t^-},x) \eta_t(dx) \le&~ \int_\RR H_{G_f}(t,Y^{t^-},x) K(dx),
\end{align*}
that
\begin{align*}
E\left [f(Y^T)\right ] \le E\left [f(X^T)\right ],
\end{align*} 
i.e. the comparison of the path-dependent function is valid.
\end{example}

\bibliographystyle{chicago}

\end{document}